\documentclass[12pt]{amsart}
\usepackage[margin=1in]{geometry}
\usepackage[english]{babel}
\usepackage[utf8]{inputenc}
\usepackage{subcaption}
\usepackage{amsmath}
\usepackage{amssymb}
\usepackage{amsfonts}
\usepackage{amsthm}
\usepackage{mathrsfs}
\usepackage[all]{xy}
\usepackage[pdftex]{graphicx}
\usepackage{color}
\usepackage{cite}
\usepackage{url}
\usepackage{cool}
\usepackage{indent first}
\usepackage[labelfont=bf,labelsep=period,justification=raggedright]{caption}
\usepackage[english]{babel}
\usepackage[utf8]{inputenc}
\usepackage{hyperref}
\usepackage{verbatim}
\usepackage[colorinlistoftodos]{todonotes}
\usepackage{tkz-fct}
\usepackage{booktabs}
\usepackage{tikz}
\usetikzlibrary{calc}
\usepackage{multicol}
\PassOptionsToPackage{dvipsnames,svgnames}{xcolor}
\usepackage{textcomp}

\renewcommand{\qedsymbol}{$\blacksquare$}

\def\multiset#1#2{\ensuremath{\left(\kern-.3em\left(\genfrac{}{}{0pt}{}{#1}{#2}\right)\kern-.3em\right)}}

\theoremstyle{plain}
\newtheorem{thm}{Theorem}
\newtheorem{lemma}[thm]{Lemma}
\newtheorem{cor}[thm]{Corollary}

\newtheorem*{unnumberedqn}{Question}

\theoremstyle{definition}

\theoremstyle{remark}

\numberwithin{equation}{section}
\numberwithin{thm}{section}

\newtheoremstyle{named}{}{}{\itshape}{}{\bfseries}{.}{.5em}{\thmnote{#3's }#1}
\theoremstyle{named}

\newtheoremstyle{named}{}{}{\itshape}{}{\bfseries}{.}{.5em}{\thmnote{#3's }#1}
\theoremstyle{named}

\catcode`,\active

\catcode`\,12

\begin{document}

\title{On Bounds and Diophantine Properties of Elliptic Curves}
\author{Navvye Anand}
\address{Euler Circle, Mountain View, CA 94040}
\email{navvye.anand@caltech.edu}

\date{\today}

\begin{abstract}
Mordell equations are celebrated equations within number theory and are named after Louis Mordell, an American-born British mathematician, known his for pioneering research in number theory. In this paper, we discover all Mordell equations of the form $y^2 = x^3 + k$, where $k \in \mathbb Z$, with exactly $|k|$ integral solutions. We also discover explicit bounds for Mordell equations, parameterized families of elliptic curves and twists on elliptic curves. Using the connection between Mordell curves and binary cubic forms, we improve the lower bound for number of integral solutions of a Mordell curve by looking at a pair of curves with unusually high rank. 
\end{abstract}
\maketitle
\section{Introduction.} Elliptic curves are smooth, projective curves of genus one of the form $y^2 = x^3 + Ax + B$. Mordell equations, are a subset of elliptic curves and take the form $y^2 = x^3 + k$. It is well known that the number of integral solutions for a Mordell equation $E: y^2 = x^3 + k$, denoted by $N(E)$ throughout this paper, is finite. Similarly, it is well known that the group of rational points on an elliptic curve is finitely generated. 

The finite nature of $N(E)$ for Mordell equations invites a natural inquiry into the precise determination of these solutions, and in this paper, we combine various explicit upper bounds for the number of integral points on a Mordell equation to find the Mordell equations of the form $E: y^2 = x^3 + k$ with $N(E) = |k|$, and generalize this Diophantine property to find the only cases where $N(E)$ is an integral multiple of $k.$ We demonstrate the following results.
\begin{itemize}
  \item \textbf{{Including the point at infinity:}}
    \begin{itemize}
      \item There are precisely three curves for which $N(E) = |k|$. These cases correspond to $k = 3, 8, 17$.
      \item There is only one curve for which $N(E) = |2k|$, corresponding to the case $k = -1$.
    \end{itemize}
  \item \textbf{{Excluding the point at infinity:}}
    \begin{itemize}
      \item There are precisely four curves for which $N(E) = |k|$. These cases correspond to $k = -1, -2, -4, 2$.
    \end{itemize}
\end{itemize}
In order to prove the aforementioned statements, we rely heavily on the connection between binary cubic forms and Mordell curves. The main theorem is proved by bounding the number of integral points on a Mordell curve by the $3$-part of the class number of the quadratic field $\mathbb{Q}(\sqrt k),$ denoted by $h_3(k)$ and then bounding the class number of quadratic field using the explicit version of Dirichlet's class number formula. Additionally, we also find explicit bounds for the number of integral points on well defined \emph{twists} of elliptic curve, and parameterized families of elliptic curves. We improve the state of the art lower bound for number of integral solutions for families of Mordell curves by exploiting this very relation. Lastly, we propose generalizations of Diophantine properties and discuss explicit bounds in literature. \\  \\ 
Before we delve into the paper, we provide a brief description of the state of the art bounds regarding elliptic curves. \noindent 
{Helfgott and Venkatesh} \cite{helfgott2005integral} proposed a novel approach to bounding \( E(K,S) \) by invoking the best sphere-packing results given by {Kabatjanskii and Levenshtein} \cite{KabLev78}. They overcame the  sensitivity to the rank of the Mordell-Weil lattice by exploring the geometry of high-dimensional Euclidean spaces, where packing problems exhibit a weak dependence on the dimension, and thereby improved upon previous bounds on elliptic curves, breaking the $O(\mathrm{|Disc(E)}|^{0.5})$ barrier.
{Bhargava et al} \cite{bhargava2017bounds} improved upon this bound and proved  $$N(E) = O(\left| \text{Disc}(E) \right|^{0.1117... + \varepsilon}).$$ The above-mentioned bound was further improved upon
by {Alpoge and Ho} \cite{alpöge2022second}. Their main result states that $$ N(E) = 
O\left(2^{\operatorname{rank}\left(E_{A, B}\right)} \prod_{p^2 \mid \Delta_{A, B}} \min \left(4\left\lfloor\frac{\nu_{p}\left(\Delta_{A, B}\right)}{2}\right\rfloor+1,7^{2^7}\right)\right)
$$
where $\nu_{p}(n)$ is the greatest nonnegative integer such that $p^{\nu_{p}(n)} \mid n$. Since the number of primes dividing $n$ has maximal order $O((\log n) / \log \log n)$ and normal order $O(\log \log n)$, this bound considerably improves upon the one by {Bhargava et al} \cite{bhargava2017bounds}. Moreover, if we assume the rank upper bound conjecture, which states that there exists an absolute constant $c>0$ such that $\operatorname{rank}\left(E\right)<c$ for all elliptic curves, the rank contribution is proven to be negligible \cite{458342}. It's also important to note that the{ Helfgott-Venkatesh bound} \cite{helfgott2005integral}
$$
N(E) \le  \left(e^{O\left(\omega\left(\Delta_{A, B}\right)\right)} 1.33^{\operatorname{rank}\left(E_{A, B}\right)}\left(\log \left|\Delta_{A, B}\right|\right)^2 \right)
$$
where, $\omega(n)$ is the number of distinct prime factors of $n$, might be stronger than Alpöge and Ho's, depending on the prime factorization of $\Delta_{A, B}$. We also note the rather interesting bound provided by {Bennett} \cite{Bennett_2015}, which states that if $F(x, y)$ is a homogeneous cubic polynomial with integral coefficients and nonzero discriminant and $m$ is a nonzero integer, then the equation $F(x, y) = m$ possesses at most $10 \times 3^{\omega(m)}$ solutions in coprime integers $x$ and $y$ where $\omega(m)$ is the number of distinct prime factors of $m$. We now prove that there exist only finitely many Mordell equations of the form $y^2 = x^3 + k $ with $N(E) = |k|$
\begin{thm}
    There exist only finitely many elliptic curves $E: y^2 = x^3 + k$ such that $N(E)=|k|$.
\end{thm}

\begin{proof}
Let us denote the discriminant of an elliptic curve $y^2 = x^3 + ax + k$ by $\mathrm{Disc}(E)  = -16(4a^3 + 27k^2).$ Now, since $a = 0$, therefore the discriminant is simply $\mathrm{Disc}(E) = -432k^2$. As demonstrated by {Bhargava et al}. \cite{bhargava2017bounds}, the number of integral points for any elliptic curve $E$ over $\mathbb Q$ in Weierstrass form with integral coefficients is at most $O_{\varepsilon}\left( |\mathrm{Disc}(E)|^{0.1117 + \varepsilon} \right)$ which implies that $N(E) = O_{\varepsilon}\left( |-432k^2|^{0.1117 + \varepsilon} \right)$. Clearly, $\lim \limits_{k \to \infty} \frac{|k|}{N(E)} = \infty$. Hence there are only finitely many cases where $N(E)  \ge |k|$ and hence only finitely many cases where $N(E) = |k|$.

\end{proof}
However, the aforementioned bounds don't allow us to explicitly compute all $k$ such that $N(E)=|k|$.
Therefore, we turn our attention to binary cubic forms, which serve as important tools while dealing with explicit upper bounds for 
$N(E)$. 
\section{Binary Cubic Forms and Explicit Bounds for $N(E)$} \label{BinaryCubic}
\subsection{Binary Cubic Forms}
\begin{thm}
    There exists a correspondence between the set of integral solutions $S_k=\left\{\left(X_1, Y_1\right), \ldots,\left(X_{N_k}, Y_{N_k}\right)\right\}$ for the Mordell equation $Y^2 = X^3 + k$ and the set  $T_k$ of triples $(F,x,y)$ where $F$ is a binary cubic form of the shape $ax^3 + 3bx^2y + 3cxy^2 + dy^3$ with discriminant $-108k$ and with integers $x,y$ satisfying $F(x,y) =1$. Furthermore, there exists a bijection between $T_k$ and $S_k$ under the actions of $SL_2(\mathbb Z) $ and $GL_2(\mathbb Z)$.
\end{thm}

We give a brief sketch of the proof outlined by {Bennett} \cite{Bennett_2015}. 
\begin{proof}

 Let \begin{equation*}F = F(x,y) = ax^3 + 3bx^2y + 3cxy^2 + dy^3 \end{equation*} be a binary cubic form with the discriminant \begin{equation*} D_{F} = -27(a^2d^2 - 6abcd -  3b^2c^2 + 4ac^3 + 4b^3d) \end{equation*} We observe the fact that the set of the binary cubic forms of the shape $F$ is closed within the larger set of binary cubic forms of the set $Z[x,y]$ under the action of both $SL_{2}$ and $GL_{2}$. Now, describe the Hessian of the $F$ to be
\begin{equation*}
H=H_F(x, y)=-\frac{1}{4}\left(\frac{\partial^2 F}{\partial x^2} \frac{\partial^2 F}{\partial y^2}-\left(\frac{\partial^2 F}{\partial x \partial y}\right)^2\right)
\end{equation*}
and the Jacobian determinant of $F$ and $H$, a cubic form $G=G_F$ defined as
\begin{equation*}
G=G_F(x, y)=\frac{\partial F}{\partial x} \frac{\partial H}{\partial y}-\frac{\partial F}{\partial y} \frac{\partial H}{\partial x} .
\end{equation*} Now, we have 
\begin{equation*}
H / 9=\left(b^2-a c\right) x^2+(b c-a d) x y+\left(c^2-b d\right) y^2
\end{equation*}
and 
\begin{equation*}
G / 27=a_1 x^3+3 b_1 x^2 y+3 c_1 x y^2+d_1 y^3,
\end{equation*}
where
\begin{equation*}
a_1=-a^2 d+3 a b c-2 b^3, \quad b_1=-b^2 c-a b d+2 a c^2, \quad c_1=b c^2-2 b^2 d+a c d, \quad d_1=-3bcd+2c^3+ad^2.
\end{equation*}

These covariants satisfy the syzygy 
\begin{equation*}
4 H(x, y)^3=G(x, y)^2+27 D F(x, y)^2 .
\end{equation*}

Defining $D_1=D / 27, H_1=H / 9$ and $G_1=G / 27$, we get
\begin{equation*}
4 H_1(x, y)^3=G_1(x, y)^2+D_1 F(x, y)^2 .
\end{equation*}

We note that if $\left(x_0, y_0\right)$ satisfies the equation $F\left(x_0, y_0\right)=1$ and $D_1 \equiv 0 (\mod 4)$ then necessarily $G_1\left(x_0, y_0\right) \equiv 0 (\mod 2)$. We may therefore conclude that $Y^2=X^3+k$, where
\begin{equation*}
X=H_1\left(x_0, y_0\right), \quad Y=\frac{G_1\left(x_0, y_0\right)}{2} \quad \text { and } \quad k=-\frac{D_1}{4}=-\frac{D}{108} .
\end{equation*}

It follows that, to a given triple $\left(F, x_0, y_0\right)$, where $F$ is a cubic form of the shape $ax^3 + 3bx^2y + 3cxy^2 + dy^3$ with discriminant $-108 k$, and $x_0, y_0$ are integers for which $F\left(x_0, y_0\right)=1$, we can associate an integral point on the Mordell equation $Y^2=X^3+k$.
The converse of this can be proven easily by taking the covariants of the factors to be \begin{equation*}
X=\frac{G_1(1,0)}{2}=\frac{G(1,0)}{54} \text { and } Y=H_1(1,0)=\frac{H(1,0)}{9} \end{equation*}
The proof of bijection between $T_k$ and $S_k$ under the action of $GL_2(\mathbb Z)$ and $SL_2(\mathbb Z)$ is achieved by constructing a contradiction.
\end{proof}
We now state an important result (without proof) discovered by {Bennett}  \cite{b1b163cc-e48b-3585-a6ac-6a99c5137c3e}.
\begin{lemma}
    If $k$ is a nonzero integer, then the equation $$ y^2 = x^3 + k$$ has at most $10 h_3(-108k)$ solutions in integers $x,y$ where $h_3(-108k)$ is the class number of the binary cubic forms with discriminant $-108k$, which is also referred to as the 3-part of class number of the quadratic field $\mathbb{Q}(\sqrt{-108k}) = \mathbb{Q}(\sqrt{-3k})$. 
\end{lemma}
\subsubsection{{Class Number Calculations}}
We work towards proving the main result of the paper by noting that Scholz's reflection principle gives us the bound $h_3(-3k) \le h_3(k) + 1.$ Now, since we have the trivial bound $h_3(k) + 1 \le h(k) + 1$, we can bound $h_3(-3k)$ by finding explicit bounds for $h(k)$, which can be done by employing Dirichlet's version of the class number formula, as shown below: 
$$ h(k) = \begin{cases} 
\dfrac{w \sqrt{|k|}}{2\pi} L(1, x), & \text{if } k < 0; \\[1em] 
\dfrac{\sqrt{k}}{\ln \varepsilon} L(1, x), & \text{if } k > 0.
\end{cases} 
$$
where $w$ is the number of automorphisms of quadratic forms of discriminant $k$, $\varepsilon$ is the fundamental unit of the quadratic field $\mathbb{Q} (\sqrt{k})$, and $L(1, \chi)$ is the Dirichlet $L$ function $\sum_{n=1}^{\infty} \frac{\chi(n)}{n}$.
\\  \\ Now in order to achieve effective bounds, we shall divide $k$  into two cases, $k > 0 $ and $k < 0$. 
Let us define $\Delta$ to be the discriminant of a real quadratic field $\mathbb{Q}(\sqrt{k})$ such that $\Delta = \begin{cases} k & \text{ if } k \equiv 1 \pmod 4 \\ 4k & \text{ if } k \not\equiv 1 \pmod 4
    \end{cases}.$ Now, {Maohua Le (Zhanjiang)} \cite{Le1994} proved that for any $k \in \mathbb N$, where $k$ is square-free, we have $h(k) \le \left\lfloor \frac{\sqrt \Delta}{2} \right\rfloor$. Since $h_3(k) \le h(k)$, we get $h_3(k) \le \sqrt{k} + 1 \implies N(E) \le 10 (\sqrt{|k|} + 1).$ But, since we want $N(E) = |k|$, we get the trivial inequality $|k| \le 10 (\sqrt{|k|}  + 1) \implies |k| \le 119$. Now for imaginary quadratic fields, the case is a bit trickier, but luckily, we utilize a combination of bounds to achieve our desired result. We begin by noting that $$w = \begin{cases} 2 & \text{ when }k < -4 \\ 
    4 & \text{ when } k = 4 \\ 6 & \text{ when } k = -3\end{cases} \implies h(k) = \dfrac{|k|^{1/2} L(1, \chi)}{\pi} \text{ for } k < -4.$$ We now introduce  a very important lemma which introduces explicit upper bounds for $L(1, \chi)$ as discovered by Louboutin \cite{Stephane}. 

\begin{lemma}\label{Louboutin}
    Let $\chi$ be a Dirichlet character modulo $q$ with conductor $f$. Then, if $\chi$ is even
\[
|L(1,\chi)| \leq \frac{1}{2}\log f + c_1 \quad \text{with} \quad c_1 = \left(2 + \gamma - \frac{\log(4\pi)}{2}\right) = 0.023\ldots 
\]  and if $\chi$ is odd, then \[
|L(1,\chi)| \leq \frac{1}{2}\log f + c_2 \quad \text{with} \quad c_2 = \frac{\left(2 + \gamma - \log \pi \right)}{2} = 0.716.
\]

\end{lemma}Replacing $f$ with $q$ in the above-mentioned bound, we get \begin{equation*}
L(1, \chi) \leq 
\begin{cases} 
\frac{1}{2} \log q + 0.023 & \text{if $\chi$ is even}, \\
\frac{1}{2} \log q + 0.716 & \text{if $\chi$ is odd}.
\end{cases}
\end{equation*}
Now, since $h_3(-3k) \le h_3(k) + 1 \le h(k) + 1$ and $h(k) \le \frac{|k|^{1/2}}{\pi}(0.5 \log |k| +0.716),$ we have $h_3(-3k) \le \frac{|k|^{1/2}}{\pi}(0.5 \log |k| +0.716) + 1 \implies$ $N(E) \le 10 \left( \frac{|k|^{1/2}}{\pi}(0.5 \log |k| +0.716) + 1 \right).$ But since we want $N(E) = |k|$, we must have $|k| \le 10 \left( \frac{|k|^{1/2}}{\pi}(0.5 \log |k| +0.716) + 1 \right)$, which only holds till $|k| =116$. By taking the union of both our results, we realize that we only need to check cases till $|k| = 119$, which is a trivial computational task. We are now ready to state the following theorem. 
\begin{thm}
   \item \textbf{{Including the point at infinity:} }there are precisely three curves for which $N(E) = |k|$. These correspond to the cases $k =3, 8, 17$.
    \item \textbf{{Excluding the point at infinity:}} there are precisely four curves for which $N(E) = |k|$, corresponding to the cases  $k = -1, -2, -4, 2$.
\end{thm}

\begin{proof}
Manually checking the cases for $|k| \le 119$ we realize that there are only four cases where $N(E) = |k|$ (including the point at infinity), corresponding to $k = 3, 8, 17$ and only two cases where $N(E) = |k|$ (excluding the point at infinity), corresponding to $k = -1, -2, -4, 2$
\end{proof}

\begin{cor}
   \item \textbf{{Excluding the point at infinity:}} there is  no curve for which $N(E) = |2k|$. 
    \item \textbf{{Including the point at infinity:}} there is only one curve for which $N(E) = |2k|$, corresponding to the case $k = -1$
\end{cor}
We note that the above-mentioned cases are the only cases where $N(E)$ is an integral multiple of $k$.
\subsection{Explicit Bounds for $N(E)$}
Now, we utilize the above-mentioned result and give an example as to how one can compute explicit bounds for $N(E)$ utilizing the result outlined by {Bhargava et al.} Recall that $N(E) = O \left(|\mathrm{Disc}(E)|^{0.1117 + \varepsilon} \right).$ Now, let $0.1117 \ldots + \varepsilon = 0.26$, this implies that $N(E) \le C|\mathrm{Disc}(E)|^{0.26}$. Since we're dealing with Mordell equations of the form $y^2 = x^3 + k$, $\mathrm{Disc}(E) = -432k^2$. Now, we simply need to find  an absolute constant $C$ such that $C(|-432k^2|^{0.26}) > 10 \left( \frac{|k|^{1/2}}{\pi}(0.5 \log |k| +0.716) + 1\right) \forall k \in \mathbb N.$
\section{Explicit Bounds on Twists of Elliptic Curves}
In this section, we find an almost sharp explicit bound on the number of integral solutions on twists of elliptic curves. We utilize the ideas in {Duke} \cite{10.1093/imrn/rnab249} to find explicit bounds, indicating the possibility of finding interesting Diophantine properties on well defined twists of elliptic curves. 
\subsection{Defining Our Twists}
We begin by defining an elliptic curve $E: y^2 = x^3 + Ax + B$, with discriminant $\Delta = -16(4A^3 + 27B^2) > 0$ and roots $e_1 < e_2 < e_3.$ Now, let $\Omega_E$ denote the real period of $E$ such that $$ \Omega_E = \int \frac{\mathrm{d}x}{y} = \int \frac{\mathrm{d}x}{\sqrt{x^3 + Ax + B}} \quad \text{ where }y > 0.$$ For $n \in \mathbb{Z}^{+}$, let $E_n: y^2 = x^3 + n^2Ax + n^3B$ be the quadratic twist on $E.$ Finally, let $\nu_{E}(n)$ denote the number of integral points on $E_{n}^{*} (\mathbb{Z})$, a subset of $E_{N}( \mathbb{Z})$ with $\gcd(x, n) = 1$, such that  $$ 
\nu_E(n) = \# \left\{ (x, y) \in \mathbb{Z}^2 ; \, y^2 = x^3 + An^2 x + Bn^3 \, \text{where} \, \gcd(n, x) = 1 \, \text{and} \, e_1 \leq \frac{x}{n} \leq e_2 \right\}.
$$

\noindent Now, for integers $(a,b,c,d,e)$ we have $$F(x,y) = (a,b,c,d,e) = ax^4 + 4bx^3y + 6cx^2y^2 + 4dxy^3 + ey^4,$$ which represents a binary quartic form with binomial coefficients. Now let $F'$ denote the content of $F$, i.e. the $\gcd$ of the coefficients of $F$. We realize that $F' = \gcd(a, 4b, 6c, 4d, e)$ is an invariant under $\Gamma \in SL(2,\mathbb{Z})$.
 Now, according to invariant theory, the other invariants of $F$ are defined as $I = I_F := ae - 4bd + c^3$ and $J = J_F := ace + 2bd -b^2e  - d^2a  - c^3.$ We note that they are related to each other by the well known syzygy $\Delta_{F} = I_{F}^3 - 27 J_{F}^2.$ Now, for $\Delta_{F} \neq 0 $, define $\operatorname{Aut} F$ to be the group of $\Gamma $ automorphs for $\Gamma \in SL(2, \mathbb{Z}),$ by counting the trivial and non trivial automorphs, we realize that the order of $\operatorname{Aut} F $ must be either $2$ or $4$ $\implies $$\# \operatorname{Aut} F \in \left\{2,4 \right\}$. Now, for integers $I_0$ and $J_0$, define $\Delta_{F} = I_{0}^3 - 27 J_{0}^2$ such that $\mathcal{F} = \mathcal{F}(I_0, J_0) $ which may contain both primitive and non primitive forms $F$. For a collection of classes,  $F_0 \in F$ the weighted (Hurwitz) class number for binary quartic forms is defined as $$h( \mathcal F_0) = \sum_{F_{0} \in F}  \dfrac{2}{\# \operatorname{Aut} F}.$$ Finally, defining $h_{E} = h(\mathcal F_{E}^{+})$ for positive definite forms, i.e., forms such that $\Delta_{F} > 0 $ and $F(x,y) > 0 $ when $x \neq 0 $ and $y \neq 0$, we get a clever analogue for the class number formula using Ikehara's version of the Wiener–Ikehara theorem
 \begin{thm}[{Duke's Theorem}]\label{Wiener-Ikehara}
 $$
\lim_{N \to \infty} \frac{1}{\sqrt{N}} \sum_{n \le N} \nu_{E}(N)  = \frac{3 \Delta \Omega_E}{2 \pi^2 \psi(\Delta)} h_E  \quad \text{ with }\psi(n) = n \prod_{p |n} \left(1 + \frac{1}{p}\right) \text{ for } p \in \mathbb{P}. 
$$
\end{thm} 
We are now ready to state the main result of this section.
\begin{thm}\label{explicitBound}
    Let there be an elliptic curve $E$ over $\mathbb{R}$ with discriminant $\Delta$, which is isomorphic to the Legendre normal form $$E(\lambda) = x(x-1)(x - \lambda)$$ for some $\lambda$ such that $0 < \lambda < 1,$ then $$ \lim_{N \to \infty} \frac{1}{\sqrt{N}} \sum_{n \le N} \nu_E(N) \le \frac{(|\Delta|^{1/2} - 1) \cdot (0.5 \log |\Delta| +0.716)}{4 \operatorname{L} \left(1, \sqrt{1 - \lambda}\right)}$$
    where $\operatorname{L} \left(a,b \right)$ is the logarithmic mean of $\left (a,b\right) \implies \operatorname{L} \left(a,b \right) = \dfrac{b - a}{\ln b - \ln a}.$
\end{thm}
\begin{proof}
\noindent We begin by stating an important lemma, 
\begin{lemma}
    $$\frac{n^2}{\zeta(2)} < \psi(n) \cdot \varphi(n) \implies \frac{n^2}{\zeta(2) \varphi(n)} < \psi(n) \text{ where $\varphi(n) $ is the Euler totient function}.$$
\end{lemma}
\begin{proof}
    $$\frac{\psi(n) \varphi(n)}{n^2} = \prod_{p |n} \left( 1 - \frac{1}{p^2} \right) > \prod_{p \in \mathbb{P}} \left(1 - \frac{1}{p^2} \right) = \zeta(2)^{-1} \implies \frac{n^2}{\zeta(2)} < \varphi(n) \cdot \psi(n)$$
\end{proof}
\noindent Now, since we have $\varphi(n) < n - \sqrt{n}$ for composite $n$, this implies that $$\frac{n^2}{\zeta(2) \cdot \left( n - \sqrt{n} \right)} < \frac{n^2}{\zeta(2) \varphi(n)} < \psi(n),$$ which in turn implies the following result: 
$$ \lim_{N \to \infty} \frac{1}{\sqrt{N}} \sum_{n \le N} \nu_E(n) \le \frac{ h_E\left( 1 - {\Delta}^{-1/2} \right)\Omega_E}{4}.$$ We now try to calculate explicit bounds for $\Omega_E$.
We now prove the following lemma
 \begin{lemma}\label{hypergeometric}
    If $0<\lambda<1$, then the real period $\Omega\left(E_\lambda\right)$ is given by
$$
\Omega\left(E_\lambda\right)=\pi \cdot { }_2 F_1\left(\begin{array}{cc}
1 / 2, & 1 / 2 \\
& 1
\end{array} ; \lambda\right) = \frac{\pi}{\operatorname{AGM}(1, \sqrt{1 - \lambda})}.
$$
\end{lemma}
\begin{proof}
We utilize the arguments presented by  {Rouse} \cite{wfu}, and begin by noting that the map $(x,y) \mapsto \left(x + \frac{\lambda + 1}{2}, \frac{y}{2} \right)$ transforms the elliptic curve $$E(\lambda) = x(x-1)(x-\lambda) \mapsto 4(x-a)(x-b)(x-c) $$
where $$ a = -\frac{\lambda + 1}{2}, b = \frac{2\lambda - 1}{3}, c = \frac{2 - \lambda}{3}.$$ Now, {Knapp} \cite{knapp1992elliptic} shows that $$ \omega_1 = \Omega_E =\int_{a}^{b} \frac{\mathrm{d}x}{\sqrt{(x-a)(x-b)(x-c)}}.$$ We utilize this result to present the following construction which allows us to compute $\Omega_E$ in terms of the arithmetic-geometric mean $(\operatorname{AGM})$. 
\begin{equation}
   \omega_1 =  \Omega_E = \int_{a}^{b} \frac{\mathrm{d}x}{\sqrt{(x-a)(x-b)(x-c)}} = \frac{\pi}{\operatorname{AGM}(\sqrt{c-a}, \sqrt{c-b})}.
\end{equation} Now, for $\omega_1$, we make the change of variables $\sqrt{x-a} = \sqrt{b-a} \sin \theta $ to obtain $$ \omega_1 = 2 \int_{0}^{\pi/2} \frac{\mathrm{d}\theta}{\sqrt{(c-b) \sin^2 \theta + (c-a) \cos^2 \theta}}.$$ Now, define $I (r,s)$ for $0 < r <s$ by $$ I(r,s) = \int_{0}^{\pi/2}  \frac{\mathrm{d}\theta}{\sqrt{r^2 \sin^2\theta + s^2\cos^2\theta}}.$$ Our ultimate goal is to prove 
    \begin{equation}\label{3.1}
        I(r,s) = \frac{\pi}{2 \operatorname{AGM}(s,r)},   
    \end{equation}
which will immediately imply \ref{hypergeometric}. We claim that \begin{equation}\label{final} 
I(r,s) = I\left(\sqrt{rs}, \frac{r+s}2 \right)
\end{equation} is sufficient in order to prove \ref{3.1} by noting that $$I(r,s) = I(\operatorname{AGM}(s,r), \operatorname{AGM}(s,r)) \quad \text{ and } \quad I(M,M) = \frac{\pi}{2M}.$$

\noindent To prove \ref{final}, regard \(0 < r < s\) as fixed. For \(0 \leq t \leq 1\), the function
\[
\frac{s + r}{2st} + \frac{(s - r)t^2}{2st}
\]
is monotically increasing in the interval $\left[0,1\right]$. Therefore
\[
\sin \theta = \frac{2 \sin \phi}{(s + r) + (s - r) \sin^2 \phi}, \quad 0 \leq \phi \leq \frac{\pi}{2},
\]
is a legitimate change of variables, and \(\theta\) extends from 0 to \(\frac{\pi}{2}\). Now, we have
\[
I(r,s) = \int_0^{\frac{\pi}{2}} \frac{\cos \theta \,  \mathrm{d} \theta}{\cos^2 \theta \sqrt{r^2 \tan^2 \theta + s^2}}.
\]
We readily compute
\[
\cos \theta \, \mathrm{d}\theta = 2 \cos \phi \left[(s + r) - (s - r) \sin^2 \phi\right] \mathrm{d} \phi
\]
with
\[
\cos^2 \theta = \frac{\left[(s + r)^2 - (s - r)^2 \sin^2 \phi\right]}{\left[(s + r) + (s - r) \sin^2 \phi\right]^2}
\]
and
\[
\tan^2 \theta = \frac{4s^2 \sin^2 \phi}{\cos^2 \phi \left[(s + r)^2 - (s - r)^2 \sin^2 \phi\right]}
.\]
Finally, we obtain
\[
I(r,s) = \int_0^{\frac{\pi}{2}} \frac{2 \mathrm{d} \phi}{\sqrt{4r^2 \sin^2 \phi + (s + r)^2 \cos^2 \phi}},
\]which completes the proof of \ref{final}, and implies that $$\Omega_E = \frac{\pi}{\operatorname{AGM} \left(1,  \sqrt{1- \lambda} \right)}.$$ We now focus our attention towards ${ }_2F_1$ and prove the fact that \begin{equation}\label{3.4}
{{ }_2 F_1\left(\begin{array}{cc}
1 / 2, & 1 / 2 \\
& 1
\end{array} ; \lambda \right)} = \dfrac{1}{\operatorname{AGM}\left(1, \sqrt{1-\lambda} \right)}.
\end{equation}
Define $$\operatorname{K}(k) = \int_{0}^{\pi/2} \frac{\mathrm{d}\theta}{\sqrt{1 - k^2\sin^2\theta}} = \int_{0}^{1} \frac{\mathrm{d}t}{(1-t^2)(1 - k^2t^2)},$$ which is the complete elliptic integral of the first kind. We have the well known identities $$\frac{2\operatorname{K}(k)}{\pi}= {}_2F_1 \left(\frac{1}{2}, \frac{1}{2}; 1; k^2\right) \quad \text{ and } \quad \operatorname{K}(k) = \frac{\pi}{2 \operatorname{AGM}(1, \sqrt{1-k^2})},
$$  which imply \ref{3.4}, completing the proof.\footnote{We note that {Rouse} \cite{wfu} gives a more beautiful proof of the same statement, using Wallis' integration formulaes, but for our purposes such a proof is extraneous.}
\end{proof}

\noindent Now, define $\operatorname{L} \left(a,b \right)$ to be the logarithmic mean of $\left \{a,b\right\} $ such that  $\operatorname{L} \left(a,b \right) = \frac{b - a}{\ln b - \ln a},$ then $$\operatorname{AGM}\left(a,b \right) \geq \operatorname{L} \left(a, b \right)\implies \dfrac{1}{\operatorname{AGM}(1, \sqrt{1- \lambda})} \leq  \dfrac{1}{\operatorname{L} (1, \sqrt{1-\lambda})},$$ which gives us the inequality \vspace{0.1in} $$\frac{\Omega_E}{\pi} = {{ }_2 F_1\left(\begin{array}{cc}
1 / 2, & 1 / 2 \\
& 1
\end{array} ; \lambda \right)} \leq \frac{\pi}{\operatorname{L} \left(1, \sqrt{1 - \lambda}\right)},$$ which in turn implies that $$ \lim_{N \to \infty} \frac{1}{\sqrt{N}} \sum_{n \le N} \nu_E(n) \le \frac{\pi}4 \cdot \frac{ h_E\left( 1 - {\Delta}^{-1/2} \right)}{\operatorname{L} \left(1, \sqrt{1 - \lambda}\right)}.$$ Now we note that $h_E$ is counting the number of equivalence classes of a binary quartic form with weight ${2}/{g}$ where $g$ is the order of the automorphism group. Recall that $\operatorname{max} \left({2}/{g}\right) = 1,$ which implies that the weighted class number formula is bounded by the actual class number formula for binary quartic forms. Finally, we note that the class number of binary quartic forms with discriminant $\Delta$ is also the $4$-part of the class number of quadratic field $\mathbb{Q}(\sqrt{\Delta}),$ which implies $$ h_E \le h_4(\mathbb{Q}(\sqrt{\Delta})) \le h(\mathbb{Q}(\sqrt{\Delta})).$$
    Now, from Louboutin's bounds \ref{Louboutin}, we know that $h(\mathbb{Q}(\sqrt{\Delta}))   \le \frac{|\Delta|^{1/2}}{\pi}(0.5 \log |\Delta| +0.716),$ which implies that $ h_E  \le \frac{|\Delta|^{1/2}}{\pi}(0.5 \log |\Delta| +0.716).$ This gives us our required explicit bound 

$$ \lim_{N \to \infty} \frac{1}{\sqrt{N}} \sum_{n \le N} \nu_E(N) \le \frac{(|\Delta|^{1/2} - 1) \cdot (0.5 \log |\Delta| +0.716)}{4 \operatorname{L} \left(1, \sqrt{1 - \lambda}\right)}.$$
\end{proof}
We also obtain the following corollaries, which may be viewed as stronger versions of \ref{explicitBound}.
\begin{cor} 
The inequality $$\lim_{N \to \infty} \frac{1}{\sqrt{N}} \sum_{n \le N} \nu_E(n) \le \frac{ (|\Delta|^{1/2} - 1) \cdot (0.5 \log(\Delta) + 0.716) \cdot  (5 - \lambda) }{8 \pi} \cdot \left[{\log \left(\frac{4}{\sqrt{1 - \lambda}}\right)}\right]$$ holds for all $\lambda \in (0,1)$.
\end{cor}

\begin{proof}
  Set $r' = \sqrt{1 - r^2}.$ Now, {Alzer} \cite{1998MPCPS.124..309A} proved that $$\operatorname{K}(r) < {\log(4/r')} \left[1 + \frac{1}{4}(r')^2 \right].$$Now, since we know that $$\Omega_E = 2\operatorname{K}(\sqrt \lambda),$$ therefore $$2 \operatorname{K}(\sqrt{\lambda}) < 2 \cdot {\log(4/\sqrt{\lambda}^{'})} \left[1 + \frac{1}{4}(\sqrt{\lambda}^{'})^2 \right]$$ implies the desired result.
\end{proof}

\begin{cor}
    $$ \lim_{N \to \infty} \frac{1}{\sqrt{N}} \sum_{n \le N} \nu_E(N) < \frac{3(\Delta^{1/2} - 1)(0.5 \log \Delta + 0.716)}{ 2\pi (\lambda + 3)} \cdot \left[\log \frac{4}{(\sqrt{1- \lambda)}}\right]$$
\end{cor}
\begin{proof}
   {Zhang et al.} \cite{1c830260f0b34c899998c8f62fd396e3} proved that $$ \operatorname{K}(r) < \frac{3}{3 + r^2} \log \frac{4}{r'}.$$ Therefore, we have$$ 2 \operatorname{K}(\sqrt{\lambda}) < \frac{6}{\lambda + 3} \log \frac{4}{\sqrt{1 - \lambda}},$$ which implies the desired result. 
\end{proof}
We note that since $\Delta > 0$, the elliptic curve $E$ has two connected components, and  is always isomorphic to an elliptic curve of the form $E(\lambda) = x(x-1)(x- \lambda)$ for some $\lambda \in \mathbb{R}$ with $0 < \lambda <1,$ proving the widespread applicability of the above-mentioned results. We also note that \ref{explicitBound} is not a tight bound under the assumption of the {GRH} because of the following result by {Shankar and Tsimerman }\cite{shankar2023nontrivial}.
\begin{thm}\label{thm: mainquadratic}
    Let $m=4$ or $5$.
Assume the refined Birch and Swinnerton-Dyer conjecture for elliptic curves over $\mathbb Q$. Then $h_m(\mathbb Q(\sqrt{D}))=O_{\varepsilon}(D^{\frac12-\frac{\delta}2+\varepsilon})$ where $\frac12-\delta$ is the best subconvex bound we have for $L$-functions of elliptic curves over $\mathbb Q$. Since $\delta $ is at least $ \frac{25}{256}$ this implies that $h_4(\mathbb{Q} ( \sqrt{D})) = O_{\varepsilon} (D^{0.451171875 + \varepsilon} ).$
\end{thm}
Since $h_E \leq h_4(\mathbb{Q}(\sqrt{\Delta}) ),$ this implies that $ h_E = O_{\varepsilon} (\Delta^{0.451171875 + \varepsilon})$, which tells us that the bound is not tight under the assumption of {GRH}. However, making this bound explicit is a non-trivial task, which is beyond the scope of this paper. 

\section{Sharp Bounds for Another Family of Elliptic Curves}
In this section, we find the best possible bounds for the number of integral points of a family of elliptic curves. We utilize the ideas of {Pincus and Washington} \cite{pincus2024field} and define an explicit relationship between certain squares in Lucas sequences and a family of elliptic curves parameterized by the coefficients of the Lucas sequence. In particular, we prove the following Theorem. 
\begin{thm}\label{NewThm}
Let $t \neq 2$ be an integer such that the fundamental unit $\omega$ of the quadratic field $\mathbb{Q} \left( \sqrt{t^2 + 4} \right)$ is $ \left(t + \sqrt{t^2 + 4} \right)\slash 2$. Then, the elliptic curve $E :=  y^2 = (t^2 + 4)x^4 - 4$ has exactly one integral point. When $t = 2$, the elliptic curve has exactly two integral points. 
 \end{thm}
\begin{proof}
      Consider the sequence defined by $u_0 = 0, u_1 = 1$ and $u_{j+2} = tu_{j+1} + u_{j}$ with the above-mentioned restrictions on $t$. We prove the following Lemma, given by {Pincus and Washington} \cite{pincus2024field}.
\begin{lemma}\label{Bijection}
    Let $$S_1:= \left\{u_j = k^2,\ k \in \mathbb{Z} \text{ and } j = \text{odd} \right\}$$ and $$S_2:= \left \{ (x,y) \in \mathbb{Z}^2 \ \text{ such that } \ y^2 = (t^2 +4)x^4 - 4\right \}.$$ Then, there exists a bijection between $S_1$ and $S_2.$
 \end{lemma}
 \begin{proof}

We note that the point $(x,y)$ corresponds to $u_j = x^2$. First, let $\overline{\omega}=(t-\sqrt{t^2+4})/2$ and $u_j=(\omega^j-\overline{\omega}^j)/(\omega-\overline{\omega})$. Now, let $v_j=\omega^j+\overline{\omega}^j$. Then, 
$
(t^2+4)u_j^2+4(\omega\overline{\omega})^j=v_j^2.$
If $j$ is odd and $u_j$ is a square, the point $(\sqrt{u_j}, v_j)$ lies on the curve $E$. Conversely, suppose $(x, y)$ is a point on the curve $E$ with $y\ge 0$. Now comes the most crucial part of the argument. Define
$
\alpha=\dfrac{y+x^2\sqrt{t^2+4}}{2}.
$
When $t$ is odd, $x$ and $y$ have the same parity, therefore $\alpha$ must be an algebraic integer. When $t$ is even, $y$ is even and $(1/2)\sqrt{t^2+4}$ is an algebraic integer, therefore $\alpha$ must be an algebraic integer. \\ \\ \noindent The norm of $\alpha$ is $-1$, which implies that $\alpha=\omega^j$ for some odd $j$. It follows that $x^2=u_j$, and therefore the map is surjective.
     
 \end{proof}
 Since a well-defined bijection exists between $S_1$ and $S_2$, their cardinalities must also be equal. Therefore, to bound the number of integral points on $E$, we need to find perfect squares with odd indices in the Lucas sequence. To do so, we refer to a Lemma by {Cohn} \cite{https://doi.org/10.1112/plms/s3-16.1.153}.
\begin{lemma}
Define $u_j$ as above with $t \neq 2$. If $j$ is odd and $u_j = k^2$ for some $k \in \mathbb{Z}$, then $j = 1.$ When $t = 2$, $j = 1$ and $j = 7.$
\end{lemma}
 \begin{proof}\renewcommand{\qedsymbol}{}
     See {Cohn} \cite{https://doi.org/10.1112/plms/s3-16.1.153} or {Nakamula and Petho} \cite{NakamulaPethő+1998+409+422}.
 \end{proof}
 \noindent The above-mentioned Lemma implies that $|S_1| = |S_2| = 1,$ which concludes the proof. 
 \end{proof} 
 We now demonstrate the widespread applicability of \ref{NewThm} by proving the following Theorem.

\begin{thm}
    If $D = t^2 + 4$ is square free, then $\varepsilon_{D} = (t + \sqrt{t^2 + 4})/2$ is the fundamental unit of $\mathbb{Q}(\sqrt{D}) $ and $\operatorname{Norm}(\varepsilon_{D}) = -1.$
\end{thm}
\begin{proof}

Since an unit $\varepsilon$ of a real quadratic field $Q(\sqrt{D})$ ($D > 0$ square-free) is an integer whose norm is equal to $\pm 1$, $\varepsilon$ is of the form $\varepsilon = \frac{t + u \sqrt{D}}{2}$; $t \equiv u \pmod{2}$, moreover $t \equiv u \equiv 0 \pmod{2}$ for the special case of $D \equiv 2, 3 \pmod{4}$, and $(t, u)$ satisfies Pell's equation $x^2 - Dy^2 = \pm 4$ because of $\pm 1 = \operatorname{Norm}(\varepsilon) = \frac{t^2 - Du^2}{4}$.
Conversely, if a pair of integers $(t, u)$ satisfies Pell's equation $t^2 - Du^2 = -4$, then clearly $t \equiv u \pmod{2}$ and moreover $t \equiv u \equiv 0 \pmod{2}$ for the special case of $D \equiv 2, 3 \pmod{4}$. For, if we assume $t \equiv u \equiv 1 \pmod{2}$, then we have $t^2 \equiv u^2 \equiv 1 \pmod{4}$, and hence $t^2 - Du^2 = -4$ implies $D \equiv 1 \pmod{4}$. Therefore, $\varepsilon = \frac{t + u \sqrt{D}}{2} = \frac{t + \sqrt{t^2 + 4}}{2}$ is a unit of $\mathbb{Q}(\sqrt{D})$ satisfying $\operatorname{Norm}(\varepsilon) = -1$. Therefore, in the special case of $y = u = 1$, we have $t^2 - D = -4$, which certainly implies that $\varepsilon_D = (t + \sqrt{t^2 + 4})/2$ is the fundamental unit provided $D = t^2 + 4$ is square-free.
\end{proof}

The Theorem implies that \ref{NewThm} holds for values of $t$ whenever $t^2 + 4$ is square-free. We note that such integers have a positive density in the set of integers, which was shown by {Estermann} \cite{Estermann1931} in 1931. Now, we prove that Theorem \ref{NewThm} holds for almost all primes.
\begin{thm}
  Theorem \ref{NewThm} holds for almost all prime values of $t.$
\end{thm}
\begin{proof}
 Begin by defining the Lucas sequence to be the sequence with the following conditions. $L_0 = 2, L_1 = 1$ and $L_n = L_{n-1} + L_{n-2}$ for $ n > 1.$ Now, generalizing Binet's formula to the Lucas sequence, we get the following relation. $$L_n = \varphi^{n} + (1 - \varphi)^n =  \left( \frac{1 + \sqrt{5}}{2}\right)^n + \left(\frac{1- \sqrt{5}}{2} \right)^n.$$ Let $P_1 :=  \left \{  \text{p: primes such that $p = L_{2n+1}$}, n \geq1 \right \}$. The infinitude (or lack thereof) of $P_1$ is an open problem, but we shall show that that there are infinitely many primes $p \notin P_1.$ For any $N > 0,$  let $$\rho(N) = \text{ number of primes $p$ such that $p \in P_1$ and $p \le N$},$$ and $$ \pi(N) =  \text{ number of primes $p \leq N$}.$$ Now, for any $N > 0$, let $\nu$ be the real number  $ \frac{\left(\log_{\varepsilon}N - 1 \right)}2$ where $\varepsilon = \left(\frac{1 + \sqrt{5}}{2} \right)$. Let $n$ denote the only integer such that $n  \leq \nu < n+1$ holds. Then, $ \nu_{2n+1}$  satisfies the inequalities $$ \nu_{2n+1}  = \varepsilon^{2n+1} + \overline{\varepsilon}^{2n+1} < \varepsilon^{2n+1} \leq \varepsilon^{2 \nu + 1} = N < \varepsilon^{2(n+1) + 1}.$$ Now, note that if $\varepsilon^{2n+1} + \overline{\varepsilon}^{2n+1}$ is prime, then $\varepsilon + \overline{\varepsilon}$ and $2n+1$ are both prime. This implies that $$ \rho(N) \leq  \pi(2N+1) \leq \pi(2 \nu + 1)  = \pi(\log_{\varepsilon} N).$$ From the PNT, we have $$\pi(n) \sim \frac{n}{\log(n)} \implies \pi(\log_{\varepsilon}) \sim \frac{\log_{\varepsilon} N}{\log \log_{\varepsilon} N}.$$ Lastly, we have 
 $$ 0 < \lim_{N \to \infty} \frac{\rho(N)}{\pi(N)} \leq \lim_{N \to \infty} \frac{\pi(\log_{\varepsilon} N)}{\pi(N)} \leq \frac{1}{\log \varepsilon } \lim_{N \to \infty} \frac{(\log N)^2}{N \log \log_{\varepsilon} N} \leq \frac{1}{\log \varepsilon} \lim_{N \to \infty} \frac{\log(N)^2}{N}= 0.
$$ Therefore, we have $$ \lim_{N \to \infty} \frac{\rho(N)}{\pi(N)} = 0,$$
which implies that there are infinitely many primes $p \notin P_1$. Now, we have the following Lemma by {Katayama} \cite{Katayama1991OnFU}.
\begin{lemma}
    For any prime $p \notin P_1$, $\omega = \left(p + \sqrt{p^2 + 4}\right)\slash2$ is the fundamental unit of the real quadratic field $\mathbb{Q} (\sqrt{p^2 + 4}) $.
\end{lemma}
This shows that \ref{NewThm} holds for almost all primes.
\end{proof}  
\begin{cor}
    Let \ $t^2 + 4 = dz^2$ for $d \in \left\{1, 2, 3,6 \right\}$. Furthermore, let $\omega = \left( t + \sqrt{t^2 + 4}\right)/2$ be the fundamental unit of the quadratic field $\mathbb{Q}(\sqrt{t^2 + 4})$. Then, the elliptic curve $E:= y^2 =  d^2(t^2 + 4)x^4 - 4$ has no integral solution. 
\end{cor}
\begin{proof}\label{StrongBijection}
Utilizing the arguments in Lemma \ref{Bijection}, we realize that there exists a bijection between $$S_3:= \left\{u_j = dk^2,\ k \in \mathbb{Z} \text{ and } j = \text{odd} \right\}$$ and $$ S_4 := \left\{(x,y) \in \mathbb{Z}^2 \text{ such that } y^2 = d^2(t^2 + 4)x^4 - 4 \right\}.$$ The result then follows from the work of {Nakamula and Petho} \cite{NakamulaPethő+1998+409+422}. 
\end{proof}
\section{Lower Bounds for Number of Integral Points}
In this section, we improve upon the state of the art lower bound for the number of integral solutions on Mordell curves by utilizing an argument of {Silverman} \cite{Silverman1983}. Let $N_{F}(m)$ denote the number of integral solutions for the binary cubic form $F(x,y) = m.$ We prove that $N_F(m) > c \log (m)^{17/19}$ for a constant $c$ and for infinitely many integers $m.$ This improves upon the previous state of the art bound, which was given by {Stewart}\cite{numdam} and states that $N_F(m) > c \log (m)^{0.85714\ldots}.$

We utilize the arguments given in \cite{numdam} and begin by proving an important lemma which will lend us to the result almost directly.
\begin{lemma}[{Silverman's Theorem}]\label{Silverman's Theorem}
    Let $F$ be a binary cubic form with integer coefficients and non zero discriminant. Let $m_0 $ be a non-zero integer such that the curve $E : F(x,y) = m_0 z^3$ has a point over $\mathbb{Q}$. Using that point as origin, we give the $F$ the structure of an elliptic curve. Denote $r$ to be the Mordell-Weil rank of the elliptic curve over $\mathbb{Q}$, then there exists a constant $c$ which depends on $F$ such that $N_{F} > c \log{(m)}^{r/r+2}$ for infinitely many integers $m.$
\end{lemma}

\begin{proof}
Let \( A = \begin{pmatrix} a & b \\ c & d \end{pmatrix} \), with \( a, b, c, d \in \mathbb{Z}\). Let \( F \) be a binary form with integer coefficients, degree $n \geq 2$ and non-zero discriminant $D(F)$. We define the binary form $F_{A}$ as 
\begin{equation*}
    F_{A}(x,y) = F(ax + by, cx + dy).
\end{equation*}
We note that 
\begin{equation*}\label{4.1}    
D(F_{A}) = (\det A)^{n \cdot n-1} D(F).
\end{equation*}
Further, for any non negative $t \in \mathbb{Z}$, we have \begin{equation*}
D(t F) = t^{(2n-2)} D(F).
\end{equation*}
Now if $A$ is equivalent to some $\Gamma \in GL(2, \mathbb Z)$, then $\det(A) = \pm 1$, and if for some $(x,y)$ we have $F(x,y) = m$, then $A(x,y) = (ax +by, cx + dy)$ is a solution of $F_{A^{-1}} (X,Y) = m$ for $(X,Y) \in \mathbb{Z}^2.$ Now we have the identity \begin{equation*}
    D({F_{A^{-1}}}) = D(F) \quad \text{ for all } A = \Gamma \in GL(2, \mathbb Z).
\end{equation*}

Lastly, we note that if $kF$ has integer coefficients for some $k$, then $D(F) \neq 0 \implies D(kF) \neq 0$, and the number of solutions of $kF(x,y) = km $ is the trivially the same as the number of solutions of $F(x,y) =m.$  
\end{proof}
We now turn our attention towards proving the main theorem  of this section, which utilizes the correspondence between binary cubic forms and Mordell equations, as highlighted before in \ref{BinaryCubic}.

\begin{thm}
Let $r$ be a positive integer which is the rank of a Mordell-Weil group of rational points of the elliptic curve $E: y^2 = x^3 + D$. Then, there exist infinitely many inequivalent binary cubic forms $F$ with integer coefficients, content $1$ and non zero discriminant for which there is a positive number $c$ depending on $F$ such that $N_{F}(m) > c \log(m)^{r/r+2}$ for infinitely many integers $m.$
\end{thm} 
\begin{proof}

We utilize a proof of a flavour similar to the one in \ref{BinaryCubic}. In particular, we utilize the syzygy between the covariants of a binary cubic form, adapting it slightly to fit our model. 

Let $P= (s,t)$ be a rational point on $E$ such that $st \neq 0$, then $$F(x,y) = x^3 - 3sx^2y-4Dy^3,$$ with discriminant $D(F) = -432Dt^2$. Let $C$ be the curve $C: F(x,y) = 1/2t$ which is non singular since $st \neq 0.$ We now set $Q  = (-s/t, -1/2t)$ as a rational point on $C$. Taking $Q$ as origin, we realize that $C$ takes the form of an elliptic curve. Now, utilizing the syzygy for the covariants of a binary cubic form, we get $$4 H(x, y)^3=G(x, y)^2+27 D F(x, y)^2.$$ Modifying this equation slightly, we get $$(4G)^2 = (4H)^3 + (432t)^2 DF^2$$ where $$H(x, y) = 9(s^2 x^2 + 4D xy - 4s Dy^2)$$ and $$G(x, y) = 54 \left((s^3 + 2D)x^3 - 6sD x^2 y + 12s^2 D x y^2 + 8D^2 y^3 \right).$$
Now that we have our syzygy sorted, we wish to utilize it somehow to get a possible relation between our curve $C$ and an elliptic curve $E$ using the covariants. This can be achieved by cleverly defining $C: F(x,y) = z^3/2t$ and $E: zy^2 = x^3 + Dz^3$. We now define the mapping  $$ \lambda: C \mapsto E \textrm{ where } \lambda([x, y, z]) = \left(\left[ z H(x, y) / 9, G(x, y) / 54, z^3 \right]\right).$$ The most crucial part of the argument lies here, where we wish to show that $\lambda$ is a non-constant morphism, and is therefore an isogeny with a degree, which allows us to effectively preserve the Mordell-Weil rank. We thus prove the following lemma: 

\begin{lemma}
    The mapping $\lambda$ is a non constant morphism, and therefore is an isogeny relating $C \mapsto E$ while preserving the Mordell-Weil rank. 
\end{lemma}
\begin{proof}
    When $z \neq 0 \text{ or } G(x,y) \neq 0$, we realize that $\lambda$ is regular. Now, when $z = 0 \text{ and } G(x,y) = 0,$  then $F(x,y) =0 $ and by our syzygy, $H(x,y) = 0$. We now realize that $\mathrm{Res}( \frac{H(X, Y)}{9}, \frac{F(X,Y)}{1}) \neq 0$, and hence $\lambda$ is a non constant morphism. This implies that $\lambda$ is an isogeny between the elliptic curves $C$ and $E$ such that the rank over $\mathbb{Q}$ is preserved. 
    \end{proof}
Furthermore, we note that if $Q$ is the origin of the elliptic curve $C$, then $\lambda(Q) = [s, -t, 1]$ is the orgin of the elliptic curve $E$. Now, the last part of proof hinges upon our ability to find binary cubic forms of the same rank as $F(x,y)$, and to prove that there are infinitely many such forms. Let $s = s_1/s_2$ and $t = t_1/t_2$ such that $\gcd(s_1,s_2) = 1$ and $s_2 > 0$ and $\gcd(t_1, t_2) = 1$ with $t_2 > 0$. We now associate $F'$ to be the binary cubic form associated with $F$ via a rank preserving argument, and then show that there can be infinitely many inequivalent forms $F'$. 
\begin{lemma}
    There are infinitely many inequivalent $F'$ such that $\operatorname{rank}F' = \operatorname{rank}F$ with content $1$.
\end{lemma}
\begin{proof}\renewcommand{\qedsymbol}{}

Set $(s,t) = (s_1/s_2, t_1,t_2)$ where $\gcd(s_1,s_2) = \gcd(t_1,t_2) = 1$ and $s_2 \neq 0$ and $t_2 \neq 0.$ We now set $b = s_2/(3,s_2)$ and get $F'(x,y) = b F(x,y).$ Finally, we show that the $\Delta(F') = -432b^4 t^2D, $ then we can simply associate the curve $C_1: F'(x,y) = m_0 z^3$ where $m_0 \neq 0$ and where the content of $F'$ is $1.$ Now utilizing \ref{Silverman's Theorem}, we get $N_{F} (m) > c \log(m)^{r/r+2}$ for infinitely many positive integers $m.$ Lastly, the fact that there are infinitely many inequivalent forms $F'$ can be proved using the fact that there can exist forms $F'$ with discriminants of arbitrarily large absolute value with points associated on $E$, completing our proof.
\end{proof}
\end{proof}
We are now ready to improve upon the previous state of the art lower bound, and utilize the above-mentioned argument to prove the main theorem of this section. 
\begin{thm}
    $N_{F}(m) > c \log{(m)}^{17/19}$ for a constant $c$ for infinitely many $m.$
\end{thm}

\begin{proof}

{Elkies}\cite{442487} found a $3$ isogenous pair of Mordell curves $y^2 = x^3 - b, y^2 = x^3 - 27b$ with $\mathrm{rank} = 17$ for $b = -908800736629952526116772283648363.$ Setting $r = 17$ in \ref{Silverman's Theorem}, we get $N_{F}(m) > c \log{(m)}^{17/19}.$
\end{proof}
We note that if the rank $r$ is unbounded for Mordell curves, i.e., as $r \to \infty$ then we have the bound $N_{F}(m) > c \log{m}$ for infinitely many $m.$ We also obtain the following corollary. 

\begin{cor}
    Let $E:= y^2 = x^3 + D$ be a Mordell curve, then $N(E) > c \log{(|D|)}^{17/19}$ for an absolute constant $c$ and for infinitely many $D.$
\end{cor}
This improves upon the bound $N(E) > c \log{(|D|)}^{11/13}$ found in {Solt and Janfada} \cite{lower}.

\section{Future Work}
In the future, we hope to generalize the Diophantine properties of Mordell equations found in this paper to the general family of elliptic curves in Weierstrass form. A natural extension of the work in this paper leads to the following question.
\begin{unnumberedqn}
    Find all elliptic curves $E : y^2 = x^3 + Ax + B$ where $(A,B) \in \mathbb{N}^2$ such that $N(E) = A + B.$
\end{unnumberedqn}
We note that the case when $A$ and $B$ have opposite signs and differ by at most $2$ is not particularly interesting, and therefore we restrict $A, B > 0$. We also note that {explicit} bound bashing for the general case of elliptic curves is more difficult than for the unique subset of Mordell curves. However, an encouraging preliminary search reveals two curves, $y^2 = x^3 + x + 3$ and $y^2 = x^3 + 3x + 5$ that have $4$ and $8$ solutions respectively. Moreover, we can aim to solve the generalized case for a subset of elliptic curves via a theorem of {Silverman} \cite{Silverman+1987+60+100}.
\begin{thm}
Let $E: y^2 = x^3 + Ax + B$ be a quasi-minimal elliptic curve, i.e. $\gcd(a^3, b^2)$ is $12$th power free, then if $\operatorname{rank}(E) = 1$ and $j$-invariant $j(E)=\frac{1728 \times 4A^3}{4A^3 + 27B^2} \in \mathbb{Z}$, then $N(E) \leq 3.28 \times 10^{33}.$ 
\end{thm}
Additionally, work done by {Hajdu and Herendi} \cite{HAJDU1998361} gives us an explicit bound on the \emph{size} of integral points on an elliptic curve. Their main result is reproduced below. 
\begin{thm}
    Let $f(x) = x^3 + ax + b$ be a polynomial with coefficients in $\mathbb{Z}$ and with non-zero discrimnant $\Delta_f.$ Then, all solutions $(x,y) \in \mathbb{Z}^2$ of the equation $y^2 = x^3 + ax + b$ satisfy\[
\max\{|x|, |y|\} \leq \exp\left\{5 \cdot 10^{64} c_1 \log(c_1 + \log(c_2))\right\}
\]
\text{with}
\[
c_1 = \frac{32|\Delta f|^{1/2} \left(8 + \frac{1}{2} \log(|\Delta f|)\right)^4}{3}, \quad c_2 = 10^{4} \max\{16a^2, 256|\Delta f|^{2/3}\}.
\]
\end{thm}

We also note a theorem of {Alpöge and Ho} that might prove useful in finding explicit bounds if we consider curves with rank $1.$
\begin{thm}
Fix $C=7^{2^7}$. Let $K$ be a number field, and let $\mathcal{O}_K$ denote its ring of integers. Let $A, B \in \mathcal{O}_K$ such that $\Delta_{A, B}:=-16\left(4 A^3+27 B^2\right) \neq 0$. Let $S$ be a finite set of places of $K$ containing all infinite places and all primes $\mathfrak{p}$ for which $\mathfrak{p}^2 \mid \Delta_{A, B}$, and let $\mathcal{O}_{K, S}$ denote the ring of $S$-integers in $K$, and let $\mathrm{Cl}(R)$ denote the class group of the ring $R$.

Let $\mathcal{E}_{A, B}: y^2=x^3+A x+B$ be an affine Weierstrass model of the elliptic curve $E_{A, B}$ over $K$. Then we have the bound
$$
\left|\mathcal{E}_{A, B}\left(\mathcal{O}_{K, S}\right)\right| \leq 2^{\operatorname{rank} E_{A, B}(K)} C^{2|S|+1}\left|\mathrm{Cl}\left(\mathcal{O}_{K, S}\right)[2]\right| .
$$
\end{thm}
There are several weaker but simpler variants of the bound, by taking $S$ to be as small as possible, i.e., the union of the infinite places and the primes $\mathfrak{p}$ with $v_{\mathfrak{p}}\left(\Delta_{A, B}\right) \geq 2$, {Alpöge and Ho} obtain the following bound on integral points:
$$
\left|\mathcal{E}_{A, B}\left(\mathcal{O}_K\right)\right| \leq 2^{\operatorname{rank} E_{A, B}(K)} C^{2[K: \mathbb{Q}]+2 \omega \geq 2\left(\Delta_{A, B}\right)+1}\left|\mathrm{Cl}\left(\mathcal{O}_K\right)[2]\right|.
$$

\section{Acknowledgements}
The author would like to thank Dr. Simon Rubinstein Salzedo, Dr. Steven Joel Miller, Parth Chavan, Khyathi Komalan and Jiwu Jang for their valuable inputs. The author also like to thank the Spirit of Ramanujan (SOR) STEM Talent Initiative for funding his participation in Euler Circle, a math program for high school students.
\bibliographystyle{vancouver} 
\bibliography{main}

\end{document}